\documentclass[10pt]{article}
 \usepackage{amsmath,amsfonts,amssymb,amsthm}
 \usepackage{hyperref}
 \usepackage{multirow}

 \theoremstyle{plain}
 \newtheorem{theorem}{Theorem}[section]

 \theoremstyle{definition}
 \newtheorem{definition}[theorem]{Definition}

 \theoremstyle{remark}

\title{Shi-type estimates and finite time singularities of flows of G$_2$ structures}
\author{Gao Chen}

\begin{document}
\maketitle
\abstract{In this paper, we extend Lotay-Wei's Shi-type estimate from Laplacian flow to more general flows of G$_2$ structures including the modified Laplacian co-flow. Then we prove a version of $\kappa$-non-collapsing theorem. We will use both of them to study finite time singularities of general flows of G$_2$ structures.}
\section{Introduction}

Let $M$ be a compact 7-manifold. A G$_2$ structure on $M$ is defined by a 3-form $\phi$ such that at each point there exists an element in $GL(7,\mathbb{R})$ which maps $\phi$ into
\begin{equation}e^{123}+e^{145}+e^{167}+e^{246}-e^{257}-e^{347}-e^{356},\end{equation}
where $e^{ijk}=e^i\wedge e^j\wedge e^k$ and $\{e^i\}$ are the standard basis of $T^*M$.
It induces a metric $g$ by
\begin{equation}g(u,v)\mathrm{Vol}_g=\frac{1}{6}(u\lrcorner\phi)\wedge(v\lrcorner\phi)\wedge\phi.\end{equation}
If $\phi$ is closed, it is called a closed G$_2$ structure. If $\psi=*\phi$ is closed, it is called a co-closed G$_2$ structure. For a G$_2$ structure, the torsion tensor $\mathbf{T}$ is defined by
\begin{equation}\nabla_a\phi_{bcd}=\mathbf{T}_a\,^e\psi_{ebcd}.\end{equation}
If the torsion tensor $\mathbf{T}$ vanishes, then it is called a torsion-free G$_2$ structure. The holonomy group of the metric induced by a G$_2$ structure is contained in G$_2$ if and only if it is torsion-free.

In order to get general existence results for the torsion-free G$_2$ structures, many versions of flows have been introduced. For example, Bryant \cite{Bryant} proposed the Laplacian flow of closed $G_2$ structures:
\begin{equation}\frac{\partial}{\partial t}\phi=\Delta_{\phi}\phi.\end{equation}
As an analogy, Karigiannis, McKay and Tsui \cite{KarigiannisMcKayTsui} proposed the Laplacian co-flow:
\begin{equation}\frac{\partial}{\partial t}\psi=\Delta_{\psi}\psi.\end{equation}
However, it is not parabolic. So Grigorian \cite{Grigorian} proposed a modified version:
\begin{equation}\frac{\partial}{\partial t}\psi=\Delta_{\psi}\psi+2\mathrm{d}((A-\mathrm{Tr}\mathbf{T})\phi),\end{equation}
where $A$ is a suitable constant.

There may be other important flows of G$_2$ structures. In general, they should satisfy the equation
\begin{equation}\frac{\partial}{\partial t}\phi_{ijk}=\frac{1}{2}h^l_i\phi_{ljk}\mathrm{d}x^i\wedge\mathrm{d}x^j\wedge\mathrm{d}x^k+\frac{1}{6}
X^l\psi_{lijk}\mathrm{d}x^i\wedge\mathrm{d}x^j\wedge\mathrm{d}x^k,
\label{Definition-1}\end{equation}
where $X$ is a vector field and $h$ is a symmetric tensor.
According to Karigiannis \cite{Karigiannis}, the equivalent equation for $\psi$ is
\begin{equation}\begin{split}\frac{\partial}{\partial t}\psi_{ijkl}=h^m_i\psi_{mjkl}+h^m_j\psi_{imkl}+h^m_k\psi_{ijml}+h^m_l\psi_{ijkm}\\-X_i\phi_{jkl} +X_j\phi_{ikl}-X_k\phi_{ijl}+X_l\phi_{ijk}.\end{split}\end{equation}

The induced equations for the metric and torsion tensor are \cite{Karigiannis}
\begin{equation}\frac{\partial}{\partial t}g_{ij}=2h_{ij},\end{equation}
and
\begin{equation}\frac{\partial}{\partial t}\mathbf{T}_{ij}=\mathbf{T}_{ij}h^{j}_m+\mathbf{T}_{ij}X^l\phi_l\,^{j}\,_m+(\nabla^k h^l_i)\phi_{klj}+\nabla_iX_j.\end{equation}

In this paper, in order to make sure the general flows make sense, we require that
\begin{equation}\frac{\partial}{\partial t}g_{ij}=2h_{ij}=-2R_{ij}+C+L(\mathbf{T})+\mathbf{T}*\mathbf{T},
\label{Definition-2}\end{equation}
\begin{equation}X=C+L(\mathbf{T})+L(\mathrm{Rm})+L(\nabla\mathbf{T})+\mathbf{T}*\mathbf{T},
\label{Definition-3}\end{equation}
and
\begin{equation}\begin{split}\frac{\partial}{\partial t}\mathbf{T}_{ij}=\Delta\mathbf{T}_{ij}+L(\mathbf{T})+L(\nabla\mathbf{T})+\mathrm{Rm}*\mathbf{T}\\
+\nabla\mathbf{T}*\mathbf{T}+\mathbf{T}*\mathbf{T}+\mathbf{T}*\mathbf{T}*\mathbf{T},\end{split}
\label{Definition-4}\end{equation}
where $L$ denote linear maps and * denote multi-linear maps. Note that in this paper, we view $\phi$, $\psi$, $g$ as constants. For example, $\mathbf{T}^{kl}\psi_{ijkl}$ is considered as $L(\mathbf{T})$.
Therefore, we have formulas like \begin{equation}\nabla(L(\mathbf{T}))=L(\nabla\mathbf{T})+\mathbf{T}*\mathbf{T}.\end{equation}
\begin{definition}
In this paper, we call a flow of G$_2$ structures reasonable if it satisfies equations (\ref{Definition-1}),(\ref{Definition-2}),(\ref{Definition-3}),(\ref{Definition-4}), the short time existence and the uniqueness.
\end{definition}
For example, for Laplacian flow \cite{LotayWei}, $X=0$, and
\begin{equation}\frac{\partial}{\partial t}g_{ij}=-2R_{ij}-\frac{2}{3}|\mathbf{T}|^2g_{ij}-4\mathbf{T}_{i}\,^k\mathbf{T}_{kj}.\end{equation}
The condition for the torsion is also satisfied.

For the modified Laplacian co-flow \cite{Grigorian}, $X=\nabla\mathrm{Tr}\mathbf{T}$, and
\begin{equation}\frac{\partial}{\partial t}g_{ij}=-2R_{ij}+\mathbf{T}^{km}\mathbf{T}^{ln}\phi_{ikl}\phi_{jmn}+(4A-2\mathrm{Tr}\mathbf{T})\mathbf{T}_{ij}.\end{equation}
The condition for the torsion is also satisfied.

The short time existence and uniqueness of the Laplacian flow were proved by Bryant-Xu \cite{BryantXu}. The analogous results for the modified Laplacian co-flow were proved by Grigorian \cite{Grigorian}.

In the case of Laplacian flow, Lotay and Wei \cite{LotayWei} proved a global version of Shi-type estimate with respect to $(|\mathrm{Rm}(p,t)|_{g(t)}^2+|\nabla \mathbf{T}(p,t)|^2_{g(t)})^{\frac{1}{2}}.$ It is equivalent to $(|\mathrm{Rm}(p,t)|_{g(t)}^2+|\mathbf{T}(p,t)|^4_{g(t)}+|\nabla \mathbf{T}(p,t)|^2_{g(t)})^{\frac{1}{2}}$ in that case. The first goal of this paper is to show a local version of Shi-type estimate with respect to $(|\mathrm{Rm}(p,t)|_{g(t)}^2+|\mathbf{T}(p,t)|^4_{g(t)}+|\nabla \mathbf{T}(p,t)|^2_{g(t)})^{\frac{1}{2}}$ for all reasonable flows of G$_2$ structures including both the Laplacian flow and the modified Laplacian co-flow. Using the global Shi-type estimate, Lotay-Wei proved that
\begin{equation}\sup_{p\in M}(|\mathrm{Rm}(p,t)|_{g(t)}^2+|\mathbf{T}(p,t)|^4_{g(t)}+|\nabla \mathbf{T}(p,t)|^2_{g(t)})^{\frac{1}{2}}\ge \frac{C}{T-t},\label{Riemannian-blow-up}\end{equation} if $T$ is the maximal existence time for the Laplacian flow. For a reasonable flow of G$_2$ structures, using our Shi-type estimate, (\ref{Riemannian-blow-up}) is also true.

One may ask whether there are any estimates for the Ricci curvature, scalar curvature and torsion torsion at maximal existence time. The answer is yes. Using the Shi-type estimate and the method of Lotay and Wei, it is easy to see that
\begin{equation}\int_0^T\sup_M(|\mathrm{Ric}|+|\mathbf{T}|^2)\mathrm{d}t=\infty.\end{equation}
In order to get better estimates using the method of Wang in \cite{Wang}, we need a $\kappa$-non-collapsing theorem. We will show that the $\kappa$-non-collapsing theorem is true if
\begin{equation}\int_{0}^{T}(T-t)\sup_M |\mathbf{T}|^4\mathrm{d}t<\infty.\end{equation}
In that case, we will prove that
\begin{equation}\limsup_{t\rightarrow T}[(T-t)\sup_M(|\mathrm{Ric}|+|\mathbf{T}|^2)]>0,\end{equation}
and \begin{equation}\limsup_{t_0\rightarrow T}[(T-t_0)^2\sup_{t\le t_0}(1+|R|+|\mathbf{T}|^2)\sup_{t\le t_0}(|\mathrm{Rm}|+|\mathbf{T}|^2+|\nabla\mathbf{T}|)]>0.\end{equation}

In particular, if in addition,
\begin{equation}\sup_M(|R|+|\mathbf{T}|^2)=o(\frac{1}{T-t}),\end{equation}
then the singularity can not be type-I. In other words, \begin{equation}\sup_M(|\mathrm{Rm}|+|\mathbf{T}|^2+|\nabla\mathbf{T}|)=O(\frac{1}{(T-t)})\end{equation}
can not be true. Moreover, using our $\kappa$-non-collapsing theorem, we can also show that any blow-up limit near finite-time singularity must be a manifold with holonomy contained in G$_2$ and has maximal volume growth rate.

In Section 2, we prove the Shi-type estimate. In Section 3, we derive the evolution equation for Perelman's $\mathcal{W}$-functional. In Section 4, we prove the $\kappa$-non-collapsing theorem. In Section 5 we discuss the finite time singularity.

\section{Shi-type estimate}
\begin{theorem}
Let $B_{r}(p)$ be the ball of radius $r$ with respect to $g(0)$ for a reasonable flow of G$_2$ structures. Assume the coefficients in the equations (\ref{Definition-1}),(\ref{Definition-2}),(\ref{Definition-3}),(\ref{Definition-4}) are bounded by $\Lambda$. For example, in the modified Laplacian co-flow case, we assume $|A|\le\Lambda$.
If \begin{equation}|\mathrm{Rm}|+|\mathbf{T}|^2+|\nabla\mathbf{T}|<\Lambda\end{equation} on $B_{r}(p)\times [0,T]$, then \begin{equation}|\nabla^k Rm|+|\nabla^{k+1}\mathbf{T}|<C(k,r,\Lambda,T)\end{equation} on $B_{r/2}(p)\times [T/2,T]$ for all $k=1,2,3,...$
\end{theorem}

\begin{proof}
We will use the method proposed by Shi in \cite{Shi}. We start from the evolution equations for the Riemannian curvature, the torsion tensor and their higher order derivatives. It is well known \cite{ChowKnopf} that if $\frac{\partial}{\partial t}g_{ij}=2h_{ij}$,
then
\begin{equation}\begin{split}\frac{\partial}{\partial t}R_{ijk}\,^l=\nabla_i\nabla_k h_{jp}+\nabla_j\nabla_p h_{ik}
-\nabla_i\nabla_p h_{jk}-\nabla_j\nabla_k h_{ip}\\
-R_{ijk}\,^q h_{qp}-R_{ijp}\,^q h_{kp},\end{split}\end{equation}
and
\begin{equation}\frac{\partial}{\partial t}R=-2\Delta\mathrm{Tr}h+2\mathrm{div(div}h)-2<h,\mathrm{Ric}>.\end{equation}
Therefore, let the degree of $\mathbf{T}$ and $\nabla$ be 1 and the degree of $\mathrm{Rm}$ be 2, then the degree of $(\frac{\partial}{\partial t}-\Delta)\mathrm{Rm}$ is 4 but it contains no $\nabla^2\mathrm{Rm}$ or $\nabla^3\mathbf{T}$ term. The degree of $(\frac{\partial}{\partial t}-\Delta)\mathbf{T}$ is 3 but it contains no $\nabla\mathrm{Rm}$ or $\nabla^2\mathbf{T}$ term. The term $\frac{\partial}{\partial t}R-\Delta R-2|\mathrm{Ric}|^2$ is a degree 4 polynomial of $\mathrm{Ric}$, $\nabla^2\mathbf{T}$, $\nabla\mathbf{T}$ and $\mathbf{T}$ but contains no $\mathrm{Ric}*\mathrm{Ric}$ term.

On the other hand
\begin{equation}\frac{\partial}{\partial t} \Gamma^{k}_{ij}=g^{kl}(h_{il,j}+h_{jl,i}-h_{ij,l}),\end{equation}
So the degree of $(\frac{\partial}{\partial t}-\Delta)\nabla\mathbf{T}$ is degree 4 but it contains no $\nabla^2\mathrm{Rm}$ or $\nabla^3\mathbf{T}$ term.

Therefore, all the terms $(\frac{\partial}{\partial t}-\Delta)|\mathrm{Rm}|^2+2|\nabla\mathrm{Rm}|^2$, $(\frac{\partial}{\partial t}-\Delta)|\mathbf{T}|^4$ and $(\frac{\partial}{\partial t}-\Delta)|\nabla\mathbf{T}|^2+2|\nabla^2\mathbf{T}|^2$ can be bounded by
\begin{equation}\epsilon(|\nabla\mathrm{Rm}|^2+|\nabla^2\mathbf{T}|^2)
+C_\epsilon(|\mathrm{Rm}|^2+|\mathbf{T}|^4+|\nabla\mathbf{T}|^2+1)^{3/2}.\end{equation}
Choose $\epsilon=1$, then \begin{equation}\begin{split}(\frac{\partial}{\partial t}-\Delta)(|\mathrm{Rm}|^2+|\mathbf{T}|^4+|\nabla\mathbf{T}|^2+1)
\le-(|\nabla\mathrm{Rm}|^2+|\nabla^2\mathbf{T}|^2)\\
+C(|\mathrm{Rm}|^2+|\mathbf{T}|^4+|\nabla\mathbf{T}|^2+1)^{3/2}.\end{split}
\label{Evolution-of-Riemannian}\end{equation}
Similarly, for all $k=1,2,3...$, both the degree of $(\frac{\partial}{\partial t}-\Delta)\nabla^k\mathrm{Rm}$ and the degree of $(\frac{\partial}{\partial t}-\Delta)\nabla^{k+1}\mathbf{T}$ are $k+4$ but they contain no $\nabla^{k+2}\mathrm{Rm}$ or $\nabla^{k+3}\mathbf{T}$ term.
So
\begin{equation}\begin{split}(\frac{\partial}{\partial t}-\Delta)(|\nabla^k\mathrm{Rm}|^2+|\nabla^{k+1}\mathbf{T}|^2)
\le-(|\nabla^{k+1}\mathrm{Rm}|^2+|\nabla^{k+2}\mathbf{T}|^2)\\
+C(k)(\sum_{j=0}^{k}(|\nabla^j\mathrm{Rm}|^{\frac{2(k+3)}{j+2}}
+|\nabla^{j+1}\mathbf{T}|^{\frac{2(k+3)}{j+2}})+|\mathbf{T}|^{2(k+3)}+1).\end{split}\end{equation}

Let $Q=(\mu+|\mathrm{Rm}|^2+|\mathbf{T}|^4+|\nabla\mathbf{T}|^2)(|\nabla\mathrm{Rm}|^2+|\nabla^2\mathbf{T}|^2)$, then
\begin{equation}
\begin{split}
(\frac{\partial}{\partial t}-\Delta)Q =& [(\frac{\partial}{\partial t}-\Delta)(\mu+|\mathrm{Rm}|^2+|\mathbf{T}|^4+|\nabla\mathbf{T}|^2)](|\nabla\mathrm{Rm}|^2+|\nabla^2\mathbf{T}|^2)
\\
&+(\mu+|\mathrm{Rm}|^2+|\mathbf{T}|^4+|\nabla\mathbf{T}|^2)(\frac{\partial}{\partial t}-\Delta)(|\nabla\mathrm{Rm}|^2+|\nabla^2\mathbf{T}|^2)\\
&-[\nabla(\mu+|\mathrm{Rm}|^2+|\mathbf{T}|^4+|\nabla\mathbf{T}|^2)][\nabla(|\nabla\mathrm{Rm}|^2+|\nabla^2\mathbf{T}|^2)]\\
\le&-(|\nabla\mathrm{Rm}|^2+|\nabla^2\mathbf{T}|^2)^2-\mu(|\nabla^2\mathrm{Rm}|^2+|\nabla^3\mathbf{T}|^2)\\
&+C(\Lambda,\mu)(|\nabla\mathrm{Rm}|^2+|\nabla^2\mathbf{T}|^2)^{\frac{4}{3}}+C(\Lambda,\mu)\\
&+C(\Lambda)(|\nabla\mathrm{Rm}|+|\nabla^2\mathbf{T}|+1)^2(|\nabla^2\mathrm{Rm}|+|\nabla^3\mathbf{T}|).
\end{split}
\end{equation}
Choose $\mu=C(\Lambda)$ large enough so that
\begin{equation}\begin{split}C(\Lambda)(|\nabla\mathrm{Rm}|+|\nabla^2\mathbf{T}|+1)^2(|\nabla^2\mathrm{Rm}|+|\nabla^3\mathbf{T}|)\le
 \frac{1}{4}(|\nabla\mathrm{Rm}|^2+|\nabla^2\mathbf{T}|^2)^2\\
 +\mu(|\nabla^2\mathrm{Rm}|^2+|\nabla^3\mathbf{T}|^2)
\end{split},\end{equation}
then
\begin{equation}\begin{split}
(\frac{\partial}{\partial t}-\Delta)Q &\le-\frac{3}{4}(|\nabla\mathrm{Rm}|^2+|\nabla^2\mathbf{T}|^2)^2
+C(\Lambda)(|\nabla\mathrm{Rm}|^2+|\nabla^2\mathbf{T}|^2+1)^{\frac{4}{3}}\\
&\le-\frac{1}{2}(|\nabla\mathrm{Rm}|^2+|\nabla^2\mathbf{T}|^2)^2+C(\Lambda)\\
&\le-C(\Lambda)Q^2+C(\Lambda)\\
&=-C_1(\Lambda)Q^2+C_2(\Lambda).
\end{split}\end{equation}

Let $\phi$ be a cut-off function which is 0 outside $B_{r}$, and is 1 inside $B_{r/2}$.
We are done if we can find $\nu$ such that
\begin{equation}H=\frac{\nu}{\phi^2}+\frac{1}{C_1(\Lambda)t}+\sqrt{\frac{C_2(\Lambda)}{C_1(\Lambda)}}\end{equation} satisfies
\begin{equation}(\frac{\partial}{\partial t}-\Delta)H>-C_1(\Lambda)H^2+C_2(\Lambda)\end{equation}
as long as $Q\le H$.

However

\begin{equation}\frac{\partial}{\partial t} H = -\frac{1}{C_1(\Lambda)t^2}, \end{equation}
\begin{equation}H^2 \ge \frac{\nu^2}{\phi^4}+\frac{1}{C_1(\Lambda)^2t^2}+\frac{C_2(\Lambda)}{C_1(\Lambda)},\end{equation}
and
\begin{equation}\begin{split}\Delta H &= \nu\Delta\frac{1}{\phi^2} \\
&= \nu\nabla(-2\frac{\nabla \phi}{\phi^3})\\
&= \frac{\nu}{\phi^4}(-2\phi\Delta\phi+6|\nabla \phi|^2).\\
\end{split}\end{equation}

So if $C_1(\Lambda)\nu>-2\phi\Delta\phi+6|\nabla\phi|^2$ as long as $Q\le H$, we are done.

Let $\tilde g$ be the metric at time 0, let $\gamma$ be the distance to $p$ with respect to $\tilde g$. Pick a non-increasing cut-off function $\eta$ which is 0 on $[r^2,\infty)$ and is 1 on $[0,r^2/4]$. Let $\phi=\eta(\gamma^2)$. Then for the ordinary derivatives
\begin{equation}\partial_i\phi=2\eta'(\gamma^2)\gamma\partial_i\gamma,\end{equation}
\begin{equation}\partial_i\partial_j\phi=2\eta'(\gamma^2)\gamma\partial_i\partial_j\gamma
+(4\eta''(\gamma^2)\gamma^2+2\eta'(\gamma^2))\partial_i\gamma\partial_j\gamma.\end{equation}
By Hessian comparison theorem,
\begin{equation}\tilde\nabla_{ij}^2\gamma=\partial_i\partial_j\gamma-\tilde\Gamma^p_{ij}\partial_p\gamma\le C(\Lambda)\tilde g_{ij}/\gamma.\end{equation}
So
\begin{equation}\Delta\gamma=g^{ij}(\partial_i\partial_j\gamma-\tilde\Gamma^p_{ij}\partial_p\gamma)\le C(\Lambda)g^{ij}\tilde g_{ij}/\gamma+g^{ij}(\tilde\Gamma^p_{ij}-\Gamma^p_{ij})\partial_p\gamma.\end{equation}

Since $|\frac{\partial}{\partial t}g_{ij}|\le C(\Lambda)$, we see that
$C(\Lambda,T)^{-1}\tilde g_{ij}\le g_{ij}\le C(\Lambda,T)\tilde g_{ij}$.

On the other hand the degree of
\begin{equation}\frac{\partial}{\partial t} \Gamma^{k}_{ij}=g^{kl}(h_{il,j}+h_{jl,i}-h_{ij,l})\end{equation}
is 3, so it is bounded by $C(\Lambda,T)(|\nabla\mathrm{Rm}|+|\nabla^2\mathbf{T}|+1)$.
Using $Q\le H$, we see that
\begin{equation}|\frac{\partial}{\partial t} \Gamma^{k}_{ij}|\le C(\Lambda,T)(\frac{\sqrt{\nu}}{\phi}+\frac{1}{\sqrt{t}}+1).\end{equation}

So \begin{equation}\Delta\gamma\le \frac{C(\Lambda,T)}{\gamma}+C(\Lambda,T)(\frac{\sqrt{\nu}}{\phi}+1),\end{equation}
and \begin{equation}\Delta\phi\ge -C(\Lambda,T,r)(\frac{\sqrt{\nu}}{\phi}+1).\end{equation}

Therefore
\begin{equation}-2\phi\Delta\phi+6|\nabla\phi|^2\le C(\Lambda,T,r)(\sqrt{\nu}+1).\end{equation}
So if we choose $\nu=C(\Lambda,T,r)$ large enough, then
\begin{equation}C(\Lambda,T,r)(\sqrt{\nu}+1)<C_1(\Lambda)\nu\end{equation}
can be achieved. We are done for the bound of $(|\nabla\mathrm{Rm}|^2+|\nabla^2\mathbf{T}|^2)$.
Using \begin{equation}Q_k=(\mu_k+|\nabla^k\mathrm{Rm}|^2+|\nabla^{k+1}\mathbf{T}|^2)(|\nabla^{k+1}\mathrm{Rm}|^2
+|\nabla^{k+2}\mathbf{T}|^2),\end{equation}
we can get higher derivative bounds.
\end{proof}
\section{Perelman's $\mathcal{W}$ functional}

In \cite{Perelman}, Perelman introduced the $\mathcal{W}$ functional \begin{equation}\mathcal{W}(g,f,\tau)=\int_{M}[\tau(R+|\nabla f|^2)+f-n](4\pi\tau)^{-n/2}e^{-f}dg.\end{equation}
By routine calculations \cite{KleinerLott}, if $\delta g_{ij}=v_{ij}$, $\delta f=h$, $v=g^{ij}v_{ij}$, $\delta\tau=\sigma$, then
\begin{equation}
\begin{split}
&\delta\mathcal{W}=\int_{M}[(\frac{v}{2}-h-\frac{n\sigma}{2\tau})(\tau(R+2\Delta f-|\nabla f|^2)+f-n)\\
&\qquad+\sigma(R+|\nabla f|^2)+h-\tau(R_{ij}+f_{ij})v^{ij}](4\pi\tau)^{-n/2}e^{-f}dg.
\end{split}
\end{equation}

For a general geometric flow \begin{equation}\frac{\partial}{\partial t} g_{ij}=-2R_{ij}+E_{ij},\end{equation}

let $f(t,p)$ solve the backwards heat equation:

\begin{equation}
\left\{ \begin{array}{l}
         \frac{\partial}{\partial t} f=-\Delta f -R +\frac{1}{2}g^{ij}E_{ij} + \frac{n}{2\tau} + |\nabla f|^2\\
         \tau=T-t,
         \end{array}\right.
\end{equation}
where $T$ is any given real number.

Let  $\varphi_t$ be the diffeomorphism generated by the time-dependent vector fields $-\nabla f$, define $\tilde g(t)=\varphi_t^* g(t)$,
and $\tilde f(t)= \varphi_t^* f(t)$, then

\begin{equation}
\left\{ \begin{array}{l}
      \frac{\partial}{\partial t} \tilde g_{ij}= -2  \tilde R_{ij} + \tilde E_{ij} - 2 \tilde f_{ij}:= \tilde v_{ij}\\
      \frac{\partial}{\partial t} \tilde f=  - \tilde\Delta \tilde f -\tilde R +\frac{1}{2}\tilde g^{ij}\tilde E_{ij}  + \frac{n}{2\tau} := \tilde h,
      \end{array}\right.
\end{equation}
where the quantities with $\sim$ sign are just the original quantities pulled back under $\varphi_t$.  Since
\begin{equation}\mathcal{W}(g(t), f(t), \tau(t))=\mathcal{W}(\tilde g(t), \tilde f(t), \tau(t)),\end{equation} we could use the variation formula to obtain
\begin{equation}
\begin{split}
& \frac{\mathrm{d}}{\mathrm{d} t}\mathcal{W}(g(t), f(t), \tau(t))\\
& = \frac{\mathrm{d}}{\mathrm{d} t}\mathcal{W}(\tilde g(t),\tilde f(t), \tau(t))\\
& = \int \big\{ - \tau (\tilde R_{ij}+\tilde f_{ij})\tilde v^{ij} + \sigma(\tilde R+|\tilde \nabla \tilde f|^2) + \tilde h   \\
& \qquad + (\frac{\tilde v}{2}-\tilde h-\frac{n\sigma}{2\tau})\big(\tau(\tilde R
+ 2\tilde\Delta \tilde f - |\tilde \nabla \tilde f|^2) + \tilde f-n\big)\big\}(4\pi\tau)^{-\frac{n}{2}}e^{-\tilde f}\mathrm{d} \tilde g\\
& = \int \{ 2\tau (\tilde R_{ij} + \tilde f_{ij})(\tilde R^{ij} +\tilde f^{ij} -\frac{1}{2} \tilde E^{ij}) \\
& \qquad - (\tilde R+|\tilde \nabla\tilde f|^2) - \tilde \Delta \tilde f -\tilde R +\frac{1}{2}\tilde g^{ij}\tilde E_{ij} + \frac{n}{2\tau} \}(4\pi\tau)^{-\frac{n}{2}}e^{-\tilde f}\mathrm{d} \tilde g\\
& =\int \{2\tau |\tilde R_{ij} +\tilde f_{ij}|^2 - 2(\tilde R + \tilde \Delta \tilde f)
 + \frac{n}{2\tau} \\
& \qquad - \tau (\tilde R_{ij} +\tilde f_{ij} -\frac{\tilde g_{ij}}{2\tau}) \tilde E^{ij}\}(4\pi\tau)^{-\frac{n}{2}}e^{-\tilde f}\mathrm{d} \tilde g\\
& =\int \{ 2\tau |\tilde R_{ij} +\tilde f_{ij} -\frac{\tilde g_{ij}}{2\tau}|^2 - \tau (\tilde R_{ij} + \tilde f_{ij} -\frac{\tilde g_{ij}}{2\tau})\tilde E^{ij}\} (4\pi\tau)^{-\frac{n}{2}}e^{-\tilde f}\mathrm{d} \tilde g\\
& = \int \{2\tau |\tilde R_{ij} +\tilde f_{ij} -\frac{\tilde g_{ij}}{2\tau}-\frac{\tilde E_{ij}}{4}|^2 -\frac{\tau}{8} |\tilde E|^2 \}(4\pi\tau)^{-\frac{n}{2}} e^{-\tilde f}\mathrm{d} \tilde g\\
& = \int \{2\tau |R_{ij} + f_{ij} -\frac{g_{ij}}{2\tau}-\frac{E_{ij}}{4}|^2 -\frac{\tau}{8} |E|^2 \}(4\pi\tau)^{-\frac{n}{2}} e^{-f}\mathrm{d} g\\
& \ge -\frac{\tau}{8} (\sup_M |E|)^2\int_{M}(4\pi\tau)^{-n/2}e^{-f}\mathrm{d}g.
\end{split}
\end{equation}

Now we are interested in the infimum
\begin{equation}\mu(g,\tau)=\inf_{\int(4\pi\tau)^{-n/2}e^{-f}\mathrm{d} g=1}\mathcal{W}(g,f,\tau).\end{equation}
Suppose $\tau_1<\tau_2$ and $f$ achieves the infimum at $T-\tau_1$.
Then by solving
\begin{equation}\frac{\partial}{\partial t}f=-\Delta f-R+\frac{1}{2}g^{ij}E_{ij}+\frac{n}{2\tau}+|\nabla f|^2\end{equation}
backwards,
\begin{equation}\int(4\pi\tau)^{-n/2}e^{-f}\mathrm{d} g=1\end{equation} is still true for all $\tau\in[\tau_1,\tau_2]$.
So \begin{equation}\mu(g(T-\tau_2),\tau_2)\le\mu(g(T-\tau_1),\tau_1)+\frac{1}{8}\int_{\tau_1}^{\tau_2}\tau\sup_{t=T-\tau} |E|^2\mathrm{d}\tau.
\label{Quasi-monotonicity}
\end{equation}

\section{$\kappa$-non-collapsing theorem}
The original $\kappa$-non-collapsing theorem of Perelman for Ricci flow in \cite{Perelman} requires the Riemannian curvature bound. However, the definition can be modified to the following version:

\begin{definition} The Riemannian metric $g$ on $M^n$ is said to be $\kappa$-non-collapsing relative to upper bound of scalar curvature on the scale $\rho$ if for any $B_g(p,r)\subset M$ with $r<\rho$ such that $\sup_{B_g(p,r)} R_g \le r^{-2}$, we have $\mathrm{Vol}_g B_g(p,r)\ge \kappa r^n$.
\end{definition}

The $\kappa$-non-collapsing theorem relative to upper bound of scalar curvature for Ricci flow was proved by Perelman (Section 13 of \cite{KleinerLott}).  The proof can be modified to get the following theorem using the quasi-monotonicity formula (\ref{Quasi-monotonicity}) in the previous section:

\begin{theorem}
Let $\frac{\partial}{\partial t} g_{ij}=-2R_{ij}+E_{ij}$ be a geometric flow on a compact manifold $M^n$. Then there exists a positive function $\kappa$ with 4 variables such that if $0<\rho\le\rho_0<\infty$, $0<\frac{T}{2}\le t_0\le T<\infty$ and
\begin{equation}\int_{0}^{t_0}(t_0+\rho^2-t)\sup_M|E|^2\mathrm{d}t<\infty,\end{equation} then $g(t_0)$ is $\kappa(g(0),T,\rho_0,\int_{0}^{t_0}(t_0+\rho^2-t)\sup_M|E|^2\mathrm{d}t)$-non-collapsing relative to upper bound of scalar curvature on scale $\rho$.
\label{main-theorem}
\end{theorem}
\begin{proof}
Fix a cut-off function $\chi(s)$ such that $\chi(s)=1$ when $|s|\le\frac{1}{2}$, and $\chi(s)=0$ when $|s|\ge 1$.
For any $g(t_0)$-metric ball $B(p,r)$ of radius $r<\rho$ which satisfies $R(x)\le r^{-2}$ for every $x\in B(p,r)$, we can define \begin{equation}u(x)=e^{L/2}\chi(\frac{d(x,p)}{r}),\end{equation} where $L$ is chosen so that
\begin{equation}(4\pi r^{2})^{-n/2}\int_{M}u^{2}=1.\end{equation}
In particular,
\begin{equation}\mathrm{Vol}(B(p,r))\ge e^{-L}(4\pi)^{n/2}r^{n},\end{equation}
\begin{equation}\mathrm{Vol}(B(p,\frac{r}{2}))\le e^{-L}(4\pi)^{n/2}r^{n}.\end{equation}
By monotonicity of $\mu$, \begin{equation}\begin{split}
\mathcal{W}(g(t_0),u,r^2)&\ge\mu(g(t_0),r^2)\\
&\ge\mu(g(0),t_0+r^2)-\frac{1}{8}\int_{t_0+r^2}^{r^2}\tau\sup_{t=t_0+r^2-\tau} |E|^2\mathrm{d}\tau\\
&=\mu(g(0),t_0+r^2)-\frac{1}{8}\int_{0}^{t_0}(t_0+r^2-t)(\sup_M |E|^2)\mathrm{d}t\\
&\ge\mu_0-\frac{1}{8}\int_{0}^{t_0}(t_0+\rho^2-t)(\sup_M |E|^2)\mathrm{d}t\\
&=\mu_1,
\end{split}\end{equation}
where $\mu_0$ is the lower bound of $\mu(g(0),\tau)$ when $\tau\in[\frac{T}{2},T+\rho_0^2]$.
So \begin{equation}\mu_1\le\int_{M}(4\pi r^2)^{-n/2}[r^2(Ru^2+4|\nabla u|^2)+u^2(-2\ln u-n)].\end{equation}
$R<Cr^{-2}$ in $B(p,r)$, $-2u^2\ln u=-2u^2(L/2+\ln\chi)$ and
\begin{equation}|\nabla u|\le\frac{e^{L/2}}{r}|\chi'(\frac{d(x,p)}{r})|\le\frac{Ce^{L/2}}{r},\end{equation}
So \begin{equation}\mu_1\le C-L+\frac{Ce^{L}\mathrm{Vol} B(p,r)}{r^n}\le C-L+\frac{C\mathrm{Vol}(B(p,r))}{\mathrm{Vol}(B(p,\frac{r}{2}))}.\end{equation}
Thus if $\frac{\mathrm{Vol}(B(p,r))}{\mathrm{Vol}(B(p,\frac{r}{2}))}<3^{n}$, then $-L\ge C$, so $\mathrm{Vol}(B(p,r))\ge C_1r^{n}.$
Let $\kappa=\min(C_1,\frac{\omega}{2})$, then we claim that $\mathrm{Vol}(B(p,r))\ge\kappa r^n$.
Otherwise, $\mathrm{Vol}(B(p,r))<\kappa r^n,$
so $\frac{\mathrm{Vol}(B(p,r))}{\mathrm{Vol}(B(p,\frac{r}{2}))}\ge 3^{n},$
so \begin{equation}\mathrm{Vol}(B(p,\frac{r}{2}))\le 3^{-n}\kappa r^{n}\le\kappa(\frac{r}{2})^{n}.\end{equation}
We can apply the same thing for $\frac{r}{2^k}$ and obtain that
\begin{equation}\mathrm{Vol}(B(p,\frac{r}{2^k}))\le\kappa(\frac{r}{2^k})^{n},\end{equation}
which is a contradiction.
\end{proof}
\section{Finite time singularity}

Now we are ready to study the finite time singularities of reasonable flows of G$_2$ structures.

First of all, using the method of Lotay-Wei and our Shi-type estimate, we can prove the following theorem:
\begin{theorem}
If $\phi(t)$ is a solution to a reasonable flow of G$_2$ structures on a compact manifold $M^7$ in a finite maximal time interval $[0,T)$, then
\begin{equation}\sup_M(|\mathrm{Rm}|^2+|\mathbf{T}|^4+|\nabla \mathbf{T}|^2)^{\frac{1}{2}}\ge \frac{C}{T-t}\end{equation}
for some constant $C>0$.
\label{Riemannian-growth}
\end{theorem}
\begin{proof}
As Lotay-Wei did in \cite{LotayWei}, if $\sup_M(|\mathrm{Rm}|^2+|\mathbf{T}|^4+|\nabla \mathbf{T}|^2)^{\frac{1}{2}}$ is bounded, then all the higher order derivatives are also bounded. So $\frac{\partial}{\partial t}g$ and $\frac{\partial}{\partial t}\phi$ are all bounded. So they and their higher order derivatives are all bounded using the background metric $g(0)$. So we can take the smooth limit. This will violate the short-time existence assumption.

Still as Lotay-Wei, we can use the equation (\ref{Evolution-of-Riemannian})
to get the required blow-up rate.
\end{proof}

Then we can get the following estimate:
\begin{theorem}
If $\phi(t)$ is a solution to a reasonable flow of G$_2$ structures on a compact manifold $M^7$ in a finite maximal time interval $[0,T)$, then
\begin{equation}\int_0^T\sup_M(|\mathrm{Ric}|+|\mathbf{T}|^2)\mathrm{d}t=\infty.\end{equation}
\end{theorem}
\begin{proof}
If it is not true, then using the evolution equation, we can see that $\frac{\partial}{\partial t}\phi$ is bounded.
So the metric is uniformly continuous. Using the proof of Theorem 8.1 of \cite{LotayWei} as well as the Shi-type estimate, we can get a contradiction.
\end{proof}

As for the better estimates of Ricci curvature, scalar curvature and torsion tensor, we can prove the following theorem using the method in \cite{Wang}
\begin{theorem}
Let $\phi(t)$ be a solution to a reasonable flow of G$_2$ structures on a compact manifold $M^7$ in a finite maximal time interval $[0,T)$. Assume that
\begin{equation}\int_{0}^{T}(T-t)\sup_M |\mathbf{T}|^4\mathrm{d}t<\infty,\end{equation}
then \begin{equation}\limsup_{t\rightarrow T}(T-t)\sup_M(|\mathrm{Ric}|+|\mathbf{T}|^2)>0,\end{equation}
and \begin{equation}\limsup_{t_0\rightarrow T}[(T-t_0)^2\sup_{t\le t_0}(1+|R|+|\mathbf{T}|^2)\sup_{t\le t_0}(|\mathrm{Rm}|+|\mathbf{T}|^2+|\nabla\mathbf{T}|)]>0.\end{equation}
\end{theorem}
\begin{proof}
In this case, the flow is $\kappa$-non-collapsing on the scale $\sqrt{T-t}$. Using Shi-type estimate and the method of Wang in \cite{Wang}, it is easy to see that $\limsup_{t\rightarrow T}(T-t)\sup_M|\frac{\partial}{\partial t}g|>0.$
So the first estimate is immediate.

As for the second estimate, we need to show that when $T-t_0<1$,
\begin{equation}\sup_{t\le t_0}(|\mathrm{Ric}|+|\mathbf{T}|^2)\le C\sqrt{O(t_0)Q(t_0)},\end{equation}
where \begin{equation}Q(t_0)=\sup_{t\le t_0}(|\mathrm{Rm}|+|\mathbf{T}|^2+|\nabla\mathbf{T}|),\end{equation}
and
\begin{equation}O(t_0)=\sup_{t\le t_0}(1+|R|+|\mathbf{T}|^2)\le 1+100 Q(t_0).\end{equation}
We will still follow the method in \cite{Wang}.
By Theorem \ref{Riemannian-growth}, we see that $Q(t_0)\ge C(T-t)^{-1}$. So the flow is $\kappa$-non-collapsing on the scale $Q(t_0)^{-\frac{1}{2}}$. Now we re-scale the flow so that $Q(t_0)=1$. Then the harmonic radius has a lower bound. So inside a finite size of ball, the metric and all of its higher derivatives are uniformly bounded.
So $(\frac{\partial}{\partial t}-D)\mathbf{T}=0$ for some elliptic operator $D$ with bounded coefficients and higher derivatives of coefficients.
So after re-scaling, \begin{equation}\sup|\nabla^k\mathbf{T}|\le C(k)\sup|\mathbf{T}|.\end{equation}
Now
\begin{equation}\begin{split}
\frac{\partial}{\partial t}R=&-2\Delta\mathrm{Tr}h+2\mathrm{div(div}h)-2<h,\mathrm{Ric}>\\
=&2\Delta R-2\mathrm{div(div Ric)}+2|\mathrm{Ric}|^2+\frac{L(\mathrm{Ric})}{Q(t_0)}+\frac{\mathbf{T}*\mathrm{Ric}}{\sqrt{Q(t_0)}}+\mathbf{T}*\mathbf{T}*\mathrm{Ric}\\
&+L(\nabla^2(\frac{C}{Q(t_0)}+\frac{\mathbf{T}}{\sqrt{Q(t_0)}}+\mathbf{T}*\mathbf{T})).\end{split}\end{equation}
The terms $2\Delta R-2\mathrm{div(div Ric)}+2|\mathrm{Ric}|^2$ are equal to $\Delta R+2|\mathrm{Ric}|^2$ by Bianchi identity. The terms $L(\mathrm{Ric})+\mathbf{T}*\mathrm{Ric}+\mathbf{T}*\mathbf{T}*\mathrm{Ric}$ are bounded by $\frac{CO(t_0)}{Q(t_0)}|\mathrm{Ric}|\le |\mathrm{Ric}|^2+C\frac{O(t_0)}{Q(t_0)}$.
The rest terms are bounded by $C\frac{O(t_0)}{Q(t_0)}$.

For any $p$, we can pick a cut-off function $\chi$ such that it is 0 outside \begin{equation}B_{g(t_0)}(p,Q(t_0)^{-1/2})\times[t_0-\frac{1}{Q(t_0)},t_0],\end{equation}
and is 1 inside \begin{equation}B_{g(t_0)}(p,\frac{1}{2}Q(t_0)^{-1/2})\times[t_0-\frac{1}{2Q(t_0)},t_0].\end{equation}
After re-scaling, it vanishes outside $B_{g(0)}(p,1)\times[-1,0]$ and is 1 inside $B_{g(0)}(p,\frac{1}{2})\times[-\frac{1}{2},0]$.

Thus \begin{equation}\begin{split}
\int_{t=0} \chi R&=\int_{-1}^{0}(\int_M\frac{\partial}{\partial t}(\chi R))\mathrm{d}t\\
&=\int_{-1}^{0}\int_M[R(\frac{\partial}{\partial t}-\Delta)\chi+\chi(\frac{\partial}{\partial t}-\Delta)R]\mathrm{d}t.
\end{split}\end{equation}
Since the geometry is bounded, it is easy to see that
\begin{equation}\int_{B_{g(0)}(p,\frac{1}{2})\times[-\frac{1}{2},0]}|\mathrm{Ric}|^2\le C\frac{O(t_0)}{Q(t_0)}.\end{equation}
Now the Ricci curvature satisfies the equation
\begin{equation}|(\frac{\partial}{\partial t}-D)\mathrm{Ric}|\le C\frac{O(t_0)}{Q(t_0)}\end{equation}
for some elliptic operator $D$ with bounded coefficients and higher derivatives of coefficients.
Therefore, we have $|\mathrm{Ric}|^2\le C\frac{O(t_0)}{Q(t_0)}$.
Before re-scaling, it is exactly $|\mathrm{Ric}|\le C\sqrt{O(t_0)Q(t_0)}$.
\end{proof}

If in addition \begin{equation}\sup_M(|R|+|\mathbf{T}|^2)=o(\frac{1}{T-t}),\end{equation} we can also show that any blow-up limit at finite time must be a manifold with maximal volume growth rate whose holonomy is contained in G$_2$.

\begin{theorem}
Let $\phi(t)$ be a solution to a reasonable flow of G$_2$ structures on a compact manifold $M^7$ in a finite maximal time interval $[0,T)$. If
\begin{equation}\int_{0}^{T}(T-t)\sup_M |\mathbf{T}|^4\mathrm{d}t<\infty,\end{equation}
and
\begin{equation}\sup_M(|R|+|\mathbf{T}|^2)=o(\frac{1}{T-t}),\end{equation}
then there exists a sequence $t_k\rightarrow T, p_k\in M$ such that \begin{equation}Q_k=(|\mathrm{Rm}|^2+|\mathbf{T}|^4+|\nabla \mathbf{T}|^2)^{\frac{1}{2}}(p_k,t_k)\rightarrow\infty,\end{equation} and $(M,Q_k^{3/2}\phi(t_k),Q_kg(t_k),p_k)$ converges to a complete manifold $M_\infty$ with a torsion-free G$_2$ structure $(\phi_\infty,g_\infty,p_\infty)$ such that
\begin{equation}\mathrm{Vol}_{g_\infty}(B_{g_\infty}(p_\infty,r))\ge\kappa r^7\end{equation}
 for some $\kappa>0$ and all $r>0$.
\end{theorem}

\begin{proof}
First of all, we can see that
\begin{equation}\limsup_{t_0\rightarrow T}[(T-t_0)\sup_{t\le t_0}(|\mathrm{Rm}|+|\mathbf{T}|^2+|\nabla\mathbf{T}|)]=\infty.\end{equation}
So we can choose a sequence such that $(T-t_k)Q_k\rightarrow\infty$.
After re-scaling, $(|\mathrm{Rm}|^2+|\mathbf{T}|^4+|\nabla \mathbf{T}|^2)^{\frac{1}{2}}$ is bounded. Moreover, $\sup(|R|+|\mathbf{T}|^2)$ converges to 0, and the manifold is $\kappa$-non-collapsing on a scale going to infinity. In particular, we get a uniform volume lower bound in any finite scale. Therefore, by our Shi-type estimate, the G$_2$ structures converge in $C^\infty$ sense to a limit G$_2$ structure. In the limit, both the scalar curvature and the torsion tensor are everywhere 0. In other words, the limit is torsion-free. Moreover, it has maximal volume growth rate.
\end{proof}

\noindent{\bf Acknowledgement:} The author is grateful to the helpful discussions with Xiuxiong Chen, Jason Lotay and Chengjian Yao.

\end{document}